\documentclass[12pt,amstex]{amsart}

\usepackage{tikz}
\usetikzlibrary{graphs}
\usepackage{mathptmx}
\usepackage{mathrsfs}
\usepackage{verbatim}
\usepackage{url}
\usepackage[all]{xy}
\usepackage{color}

\usepackage{tikz}
\usetikzlibrary{arrows.meta,animations}
\tikzset{>=stealth}
\usetikzlibrary{calc}
\usetikzlibrary{scopes}

\usepackage[colorlinks=true,citecolor=blue]{hyperref}
\usepackage{stmaryrd}
\usepackage{epsfig}
\usepackage{amsmath,bm}
\usepackage{amssymb}
\usepackage{amssymb,amsthm,amsmath}
\usepackage{mathrsfs}
\usepackage{amscd}
\usepackage{graphicx}
\usepackage{pstricks}
\usepackage{theoremref}

\newtheorem{thm}{Theorem}[section]
\newtheorem{lem}[thm]{Lemma}

\newtheorem{ques}{Question}

\theoremstyle{definition}
\newtheorem{de}[thm]{Definition}

\theoremstyle{remark}

\numberwithin{equation}{section}
\def \N {\mathbb{N}}

\def \Z {\mathbb{Z}}

\def \F {\mathcal F}

\def \P {\mathcal P}

\def \a {\alpha }

\def \0 {\bf 0}

\begin{document}
	\title{Some negative answers to the Bergelson-Hindman's question}
	
	\author{Qinqi Wu}

\address{School of Mathematics, Shanghai University of Finance and Economics, Yangpu, Shanghai, 200433, P.R. China}
\email{wuqinqi@mail.shufe.edu.cn}



	\date{\today}
	\begin{abstract}
			Let $p_1,\dots,p_d$ be integral polynomials vanishing at $0$. It was asked by Bergelson and Hindman whenever $A$ is large, whether the set $\{(m,n)\in \N^2:m+p_1(n),m+p_2(n),\dots,m+p_d(n)\in A\}$ be large in the same sense. In this paper, we 
			give negative answers to this question when ``large'' being the notions of  ``central*'', 
			 ``IP*'', ``IP$_n$*'', ``IP$_{<\omega}$*'' and ``$\Delta$*''. 
	\end{abstract}
	
	\maketitle
	\section{introduction}
	Throughout this paper integers, nonnegative integers and natural numbers are denoted by $\Z$, $\Z_+$ and $\N$ respectively. An {\it integral polynomial} is polynomial taking integer values at the integers.
	
	
	\medskip
	
	Van der Waerden's theorem states that each piecewise syndetic subset of $\Z$ contains arbitrarily long arithmetic progressions. That is, if $A\subset\Z$ is piecewise syndetic, then for all $d\in\N$, the set $\{(m,n)\in\Z^2:m,m+n,\ldots,m+(d-1)n\in A,n\ne 0\}$ is not empty. Furstenberg and Glasner \cite{FG} obtained the following result using the Stone-\v{C}ech compactification of $\Z$.
	
	\medskip
	\noindent{\bf Theorem} (Furstenberg-Glasner) {\it 
		Let $d\in\N$ and $A$ be piecewise syndetic in $\Z$, then $$\{(m,n)\in\Z^2:m,m+n,\ldots,m+(d-1)n\in A\}$$ is piecewise syndetic in $\Z^2$.
}
	
	\medskip
	Later Beiglb\"{o}ck \cite{B} provided a combinatorial proof for the Fustenberg-Glasner's result just using van der Waerden's theorem. Bergelson and Hindman \cite{BH} extended this result to apply to many notions of largeness 
	in arbitrary semigroups and to partition regular structures other than arithmetric progressions. One of their results is,

	\medskip
	\noindent{\bf Theorem} (Bergelson-Hindman) {\it 
		Let $B\subset\N$ and $d\in\N$. Let ``large'' be any of ``piecewise syndetic'', ``central'', ``central*'', ``thick'', ``PS*'', ``IP*'', ``IP$_{<\omega}$*'', ``IP$_n$*'', or ``$\Delta$*''. If $B$ is large in $\N$, then $\{(m,n)\in\N^2:m,m+n,\ldots,m+(d-1)n\in B\}$ is large in $\N^2$.
	}
	
	\medskip
	In \cite{BH}, the authors also gave an polynomial extension of van der Waerden's theorem: Let $p_1,\dots,p_d$ be polynomials with $p_i(n)\in\N, p_i(0)=0$, if $S\subset\N$ is piecewise syndetic, then the set $\{(m,n)\in\N^2:m+p_1(n),m+p_2(n),\ldots,m+p_d(n)\in S,n\ne 0\}$ is not empty.	
	They asked a question \cite[Question 4.7]{BH} as follows

	\begin{ques}\label{BHQ}
		Let $d\in\N$, and let $p_1,p_2,\dots,p_d$ be integral polynomials with $p_i(0)=0$. For which, if any, of the notions of ``piecewise syndetic'', ``central'', ``central*'', ``thick'', ``PS*'', ``IP*'', ``IP$_{<\omega}$*'', ``IP$_n$*'', or ``$\Delta$*'', is it true that whenever $A$ is a large subset of $\N$, $\{(m,n)\in \N^2:m+p_1(n),m+p_2(n),\dots,m+p_d(n)\in A\}$ is large in the same sense?
	\end{ques}

	There is an obstacle to solve Question \ref{BHQ} directly with combinatorial methods  since the fact that $\{(m+p_1(n),m+p_2(n),\dots,m+p_d(n)):m,n\in\Z\}$ is not a semigroup when $\max\{\deg p_i\}\geq 2$.
	
	\medskip
	
	Recently, Huang, Shao and Ye \cite{HSY} confirmed the question for subsets with ``piecewise syndeticity'' by showing the density of minimal points of a dynamical system of $\Z^2$ action associated with the piecewise syndetic set $S$ and polynomials $\{p_1,\dots,p_d\}$. 
	Wu \cite{wu} confirmed the question for subsets with ``thickly syndeticity (PS*)'' also by the method of dynamical systems.

	\medskip
	
	In this paper, we study the rest of the Bergelson-Hindman's question. It is unexpected that the answer is negative for most of the other notions of size. 
	The following theorem is our main result.
	
\begin{thm}\label{main}
	Let $d\in\N$ and $p_1,\dots,p_d$ be any integral polynomials with $p_i(0)=0$.
	\begin{enumerate}
		\item If $A$ is thick in $\N$, then $\{(m,n)\in \N^2:m+p_1(n),\dots,m+p_d(n)\in A\}$ is thick in $\N^2$.
		
		\item 
		Let ``large'' be any of 
		``central*'', ``IP*'', ``IP$_{<\omega}$*'', ``IP$_n$*'', or ``$\Delta$*''. If  $\max\{\deg p_{i}\}\geq2$, then there exists a large set $A$ in $\N$ (resp. $\Z$) 
		such that the set $$\{(m,n)\in\N^2\text{ (resp. $\Z^2$)}
		:m+p_1(n),\dots, m+p_d(n)\in A\}$$ is not large.
	\end{enumerate}
\end{thm}	
	The proofs of Theorem \ref{main} are done by combinatorial arguments. 
	The question for central sets remains open, since we have no suitable way to determine whether a given set is central or not in $\N^2$. 
	
	\medskip

	The paper is organized as follows. In Section 2, we state some necessary notions and some known facts used in the paper. In Section 3, we first 
confirm Question \ref{BHQ} for   ``thick''  subsets and then we construct suitable large sets to
	prove Theorem \ref{main} (2).

	\medskip
	\noindent{\bf Acknowledgement.} The author would like to thank Professors Song Shao and Xiangdong Ye for helpful discussions and remarks.
	
	\bigskip
	
\section{Preliminary}
Given a commutative semigroup $(S,+)$. Let $\P(S)$ be the collection of non-empty subsets of $S$ and $\P_f(S)$ be the collection of finite non-empty subsets. A subset $\F$ of $\P(S)$ is a {\it family}, if it is hereditary upwards, i.e. $F_1\subset F_2$ and $F_1\in\F$ imply $F_2\subset \F$.
A family $\F$ is proper if it is a proper subset of $\P$, i.e. neither empty nor all of $\P$. If a proper family $\F$ is closed under finite intersections, then $\F$ is called a filter. For a family $\F$, the dual family is $$\F^*=\{F\in\P:S\setminus F\notin \F\}=\{F\in\P:F\cap F'\ne\varnothing \ \text{for all}\ F'\in\F\}.$$ For any $x\in S$ and $A\subset S$, define $-x+A=\{y\in S:s+y\in A\}$. Let us recall some notions related to families.
	\begin{de}
	 Let $(S,+)$ be a commutative semigroup. A subset $A$ of $S$ is called 
	 
	 \begin{enumerate}
	 	\item  {\it syndetic} if there is a finite subset $F\in\P_f(S)$ such that $S=\bigcup_{s\in F}(-s+A)$. 
	 	
	 	\item  {\it thick} if for every $F\in\P_f(S)$, there is $s\in S$ such that $Fs\subset A$. 
	 	
	 	\item {\it piecewise syndetic} if there exists some $G\in\P_f(S)$ such that for every $F\in\P_f(S)$
	 	there exists $s\in S$ such that $Fs\subset\bigcup_{t\in G} (-t+A)$.
	 	
	 	\item {\it thickly syndetic} if it has non-empty intersection with every piecewise syndetic set.
	 \end{enumerate}
	\end{de}
	It is clear syndetic*=thick and PS*=thickly syndetic, where PS is the family of piecewise syndetic subsets.

	\begin{de}
		A subset $A\subset S$ is called a {\it $\Delta$-set} if it contains  $\Delta(\{x_n\})=\{x_n-x_m: n>m\}$ for some infinite sequence $\{x_n\}_{n=1}^{\infty}$;  is called a {\it $\Delta$*-set} if it has non-empty intersection with every $\Delta$-set.
	\end{de}	

\medskip

In this paper, the operation in $S$ is denoted by ``+''. Given a sequence $\{x_i\}_{n=1}^{\infty}\subset S$. Let $$FS(\{x_n\}_{n=1}^{\infty})=\{\sum_{i\in \a}s_i: \a \ \text{is a finite non-empty subset of}\ S\}.$$

	\begin{de}
		Let $A\subset S$. The set $A$ is an {\it IP-set} if and only if there exists a sequence $\{x_n\}_{n=1}^{\infty}\subset S$ such that $FS(\{x_n\}_{n=1}^{\infty})\subset A$. $B\subset S$ is an {\it IP*-set} if and only if $B\cap A\ne\varnothing$ for every IP-set $A$.
	\end{de}

	
	The notion of central sets of $\N$ was introduced by Furstenberg \cite{F} in terms of notions of topological dynamics, and the denition makes sense in any semigroup. By a {\it dynamical system} we mean a pair $(X,\langle T_s\rangle_{s\in S})$, where $X$ is a compact metric space with a metric $\rho$, $(S,+)$ is a semigroup and $T_s:X\rightarrow X$ is a homeomorphism for each $s\in S$ with $T_s\circ T_s=T_{s+t}$ for any $s,t\in S$. 
	Let $(X,\langle T_s\rangle_{s\in S})$ be a dynamical system. A point $y\in X$ is {\it uniformly recurrent} if and only if for each neighborhood $U$ of $y$, $\{s\in S:T_s y\in U\}$ is syndetic. We say $x\in X$ proximal to $y$ if there is a sequence $\{s_n\}_{n=1}^\infty$ in $S$ such that $\rho(T_{s_n} x, T_{s_n} y)\rightarrow 0$. \begin{de} 
		A subset $A\subset S$ is a {\it central set} if there exists a dynamical system $(X,\langle T_s\rangle_{s\in S})$, a point $x\in X$, and a uniformly recurrent point $y$ proximal to $x$, and a neighborhood $U$ of $y$ such that $A\supset N(x,U):=\{s\in S:T_s x\in U\}$. $B$ is a {\it central* set} if and only if $B\cap A\ne\varnothing$ for every central set $A$. 
	\end{de}

In \cite{BH90}, central sets is shown to be equivalent to a much simpler algebraic characterization.
	We have the following implications (see \cite{F} for detailed proofs), which will be used to give counterexamples in the next sections

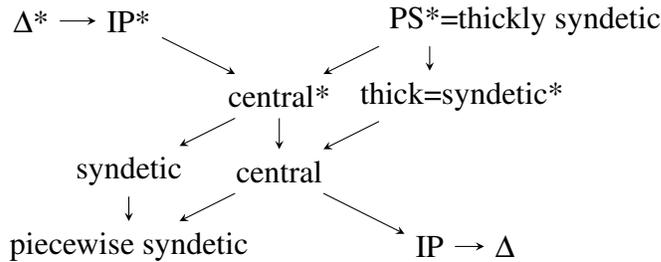
\begin{center}
	\begin{tikzpicture}[new set=import nodes]
		\begin{scope}[nodes={set=import nodes}]
			\node  (a) at (-3.3,1) {$\Delta$*};
			\node  (b) at (-2,1) {IP*};
			\node  (c) at (0,0) {central*};
			\node  (d) at (2,1) {\quad \quad \quad \quad \quad \quad PS*=thickly syndetic};
			\node  (e) at (-2,-1) {syndetic};
			\node  (f) at (-2,-2) {piecewise syndetic};
			\node  (g) at (0,-1) {central};
			\node (h) at (2,0) {\quad \quad thick=syndetic*};
			\node [black] (i) at (2,-2) {IP};
			\node [black] (j) at (3,-2) {$\Delta$};
		\end{scope}
		\graph {(import nodes);
			a -> b -> c -> e -> f;
			c->g->f;
			d->c;
			d->h->g->i->j};
	\end{tikzpicture}
\end{center}

	\begin{de}
		Let $n\geq2$. A set $A\subset S$ is an {\it $IP_n$-set} if  whenever $\F$ is a finite partition of $A$, there exist $F\in\F$ and $x_1,\dots,x_n\in S$ such that $FS(\{x_t\}_{t=1}^{n})\subset F$. A set $B\subset S$ is an {\it IP$_n$*-set} if and only if $B\cap A\ne\varnothing$ for every IP$_n$ set $A$.
	\end{de}

\begin{de}
	A set $A\subset S$ is an {\it IP$_{<\omega}$-set} if 
	 whenever $\F$ is a finite partition of $A$ and $n\in\N$, there exist $F\in\F$ and $x_1,\dots,x_n\in S$ such that $FS(\{x_t\}_{t=1}^{n})\subset F$. A set $B\subset S$ is an {\it IP$_{<\omega}$* set} if and only if $B\cap A\ne\varnothing$ for every IP$_{<\omega}$ set $A$.
\end{de}

Observe that 
$$\text{IP}\rightarrow \text{IP}_{<\omega}\rightarrow\dots\text{IP}_n\rightarrow \dots\rightarrow\text{IP}_4\rightarrow \text{IP}_3\rightarrow \text{IP}_2,$$  $$\text{IP}_2{\text{*}}\rightarrow \text{IP}_{3}{\text{*}}\rightarrow \text{IP}_{4}{\text{*}}\rightarrow\dots\rightarrow\text{IP}_{n}{\text{*}}\rightarrow\dots\rightarrow \text{IP}_{<\omega}{\text{*}}\rightarrow \text{IP}{\text{*}}.$$

\bigskip
\section{Proof of Theorem \ref{main}}
In this section, we give the proof to our main theorem.  
We will prove Theorem \ref{main} (1) first and Theorem \ref{main} (2) follows from Therorems \ref{thm-IP*}, \ref{IP_n*}, \ref{Delta*}, and \ref{central*}.

\subsection{The case for a thick set}\

\begin{thm}\label{thick}\
	Let $d\in\N$ and $p_1,\dots,p_d$ be integral polynomials
	with $p_i(0)=0$. If $A$ is thick in $\N$, then
	$$\{(m,n)\in\N^2:m+p_1(n),m+ p_2(n),\dots,m+p_d(n)\in A\}$$
	is thick in $\N^2$.
\end{thm}
\begin{proof}
	Assume that $A\subset\N$ is thick. There exists a sequence $\{a_n\}_{n\in\N}\subset\N$ such that $$[a_n,a_n+n]:=\{a_n,a_n+1,\dots, a_n+n\}\subset A.$$
	For every $N\in\N$, denote $$N_{\min}:=\min\{p_i(n):0\leq n\leq N, 1\leq i\leq d\},$$ $$N_{\max}:=\max\{p_i(n):0\leq n\leq N, 1\leq i\leq d\}.$$	
	Now we take suitable $n(N)\in\N$ and constant $b\in\N$ such that
	$$a_{n(N)}<b+N_{min}, b+N_{max}+N<a_{n(N)}+n(N).$$
	Then we have 
	$$\begin{aligned}
		[b,b+N]\times[0,N]&=\{(m,n)\in\N^2:b\leq m\leq b+N, 0\leq n\leq N\}\\
		&\subset \{(m,n)\in\N^2:m+p_1(n),m+ p_2(n),\dots,m+p_d(n)\in A\}.
	\end{aligned}$$
	So the set $\{(m,n)\in\N^2:m+p_1(n),\dots,m+p_d(n)\in A\}$ contains arbitrarily size blocks of integers, and hence, is thick.
\end{proof}

Naturally, we can consider the dual notion of thick: syndetic. It is known that  Bergelson-Hindman's theorem doesn't hold for syndetic subsets (remarked in \cite{FG}, see also \cite[Theorem 3.9]{BH}). 
For polynomials we have the following result.

\begin{thm}\label{syndetic-d=1}
	Let $p(n)$ be an integral polynomial
	with $p(0)=0$. If $A$ is syndetic in $\Z$, then
	$\{(m,n)\in\Z^2:m+p(n)\in A\}$
	is syndetic in $\Z^2$.
\end{thm}
\begin{proof}
	Let $B\subset \Z^2$ be any thick set. It is clear that $B':=\{p(n)+m:(m,n)\in B\}$ is a thick set in $\Z$, so we have $A\cap B'\ne\varnothing$. Thus,
	$$\{(m,n):p(n)+m\in A\cap B'\}=\{(m,n):p(n)+m\in A\}\cap B\ne \varnothing,$$
	which implies that set $\{(m,n)\in\Z^2:p(n)+m\in A\}\subset \Z^2$ is syndetic.
\end{proof}

\noindent {\it Remark.}

(1) Theorem \ref{syndetic-d=1} doesn't hold for $d=2$ in general, since we may choose a thick subset $B\subset\Z^2$ and integral polynomials $p_1,p_2$ such that $B'_1\cap B'_2=\varnothing$, where $B'_i:=\{p_i(n)+m:(m,n)\in B\},i=1,2.$ Thus, for any syndetic set $A$, $$\begin{aligned}
	\{(m,n)\in\Z^2:&p_1(n)+m,p_2(n)+m\in A\}\cap B\subset\\ &\{(m,n)\in\Z^2:p_1(n)+m,p_2(n)+m\in B_1'\cap B'_2\}=\varnothing.
\end{aligned}$$
It deduces that  $\{(m,n)\in\Z^2:p_1(n)+m,p_2(n)+m\in A\}$ is not syndetic.

\medskip
(2) Theorem \ref{syndetic-d=1} holds for $\N$ if and only if the first coefficient of $p(n)$ is positive.
Let $A$ be syndetic in $\N$. 
Suppose the first coefficient of $p(n)$ is positive, then $|\{n:p(n)\leq 0\}|<+\infty$. Then for cofinitely many $n\in\N$,  $A-p(n)$ is syndetic with the same gap as $A$. 
Thus, $$\{(m,n)\in\N^2:m+p(n)\in A\}=\bigcup_{n\in\N}(A-p(n))\times\{n\}$$ is syndetic. 

Otherwise, 
for any $N\in\N$, there is $n_N$ such that $\min(A-p(n))>N$ whenever $n>n_N$. Then $$\{(m,n)\in\N^2:m+p(n)\in A\}\bigcap (\bigcup_{N\in\N}[n_N+1,+\infty)\times[0,N])=\varnothing,$$
where $\bigcup_{N=1}^{\infty}[n_N+1,+\infty)\times[0,N]$ is thick in $\N^2$.

\medskip
(3) 
As we know if $|\N\cap \{p(n):n\in\Z\}|<+\infty$, then $$|\N^2\cap \{(m,n)\in \Z^2:m+p(n)\in A\subset\N\}|<+\infty$$ and the set would not be large. Thus, in the sequel, we always assume that $|\N\cap \{p(n):n\in\Z\}|=+\infty$. i.e., $a_l\geq 0$ when $p(n)=\sum_{i=1}^{l}a_i n^{i},l\geq 2$.

\medskip
\subsection{The case for IP*, IP$_n$*, IP$_{<\omega}$ sets }\
	
\medskip	
	In this subsection, we give negative answers to Question \ref{BHQ} for IP*, IP$_n$*, IP$_{<\omega}$ sets. We begin with the next lemma.
	
	
	\begin{lem}\label{IP*}
		Let $S=\{s_i:s_i<s_{i+1}\}$ be an infinite sequence of $\N$. If $s_{n+1}-s_{n}\rightarrow\infty$, then $\N\backslash S$ is an IP*-set.
	\end{lem}
	\begin{proof}
		Assume that $\N\backslash S$ is not an IP*-set. So there is an IP-set $B=FS(\{t_i\})$ such that $(\N\backslash S)\cap B=\varnothing$. i.e., $B\subset S$. In particular, $t_n,t_n+t_1\in S$ for any $n\in\N$. Hence, we have  $\lim_{n\rightarrow\infty}\inf(s_{n+1}-s_n)\leq t_1<\infty$, a contradiction.
	\end{proof}

\begin{thm}\label{thm-IP*}
	Let $d\in\N$ and $p_1,\dots,p_d$ be integral polynomials with $p_i(0)=0$ and $\max\{\deg p_{i}\}\geq2$. Then there exists an IP*-set $A$ in $\N$ such that the set $$\{(m,n)\in\N^2:m+p_1(n),\dots,m+p_d(n)\in A\}$$ is not an IP*-set in $\N^2$.
\end{thm}
\begin{proof}
	It is enough to prove the case of $d=1$ with $\deg p\geq2$ since  $\{(m,n)\in\N^2:m+p_1(n)\in A\}\supset\{(m,n)\in\N^2:m+p_1(n),\dots,m+p_d(n)\in A\}$.
	
	\medskip
	
	Without loss of generality, assume that  $p(n)=\sum_{i=1}^{l}a_{i}n^i$ with $l\geq 2$ and $a_l>0$. So there exists $N\in\N$ such that $p(n)$ is an increasing function whenever $n>2^N$.
	Now we begin to construct the IP*-set we need. 
	We take $(m_i,n_i)=(2^{iN},2^{iN})\in\N^2$.
	Denote
	\begin{align*}
		S:=&\{p(n_{i_1}+\dots+n_{i_k})+m_{i_1}+\dots+m_{i_k}:i_{1}<\dots<i_{k}, k\in\N\}\\
		\supset &\{s_n=p(n\cdot2^{N})+n\cdot 2^N:n\in\N\}.
	\end{align*}	
	It is clear that $s_{n+1}-s_n\rightarrow\infty$.
	Then by Lemma \ref{IP*}, $A=\N\backslash S$ is an IP*-set.  
	 	
	\medskip
	 
	At the same time, we have 
	\begin{align*}
		&\{(m,n)\in\N^2:p(n)+m\in A\}=\{(m,n)\in\N^2:p(n)+m\notin S\}\\
		=&\{(m,n)\in\N^2:p(n)+m\notin \{p(n_{i_1}+\dots+n_{i_k})+m_{i_1}+\dots+m_{i_k}:i_{1}<\dots<i_{k}, k\in\N\}\}\\
		\subset & \N^2\backslash FS(\{(m_i,n_i)\}_{i=1}^{\infty}).
	\end{align*}
	i.e., 
	$$FS\{(m_i,n_i)\}\cap \{(m,n)\in\N^2:p(n)+m\in A\}=\varnothing,$$
	which means that
	$\{(m,n)\in\N^2:p(n)+m\in A\}$ is not an IP*-set of $\N^2$.
\end{proof}

\medskip

	Next we consider the cases of IP$_n$* ($n\geq 2$) and IP$_{<\omega}$*.

	\begin{lem}\label{IPn*}
		Let $n\in\N$ and $S=\{s_i:s_i<s_{i+1}\}$ be an infinite sequence of $\N$. If $\sum_{i=1}^{n}s_{k_i}\notin S$ for any $1\leq k_1<\dots<k_n$, then $\N\backslash S$ is an IP$_n$*-set.
	\end{lem}
\begin{proof}
	Assume that there exists an IP$_n$*-set $B$ such that $(\N\backslash S)\cap B=\varnothing$, then we have $B\subset S$. By the definition, whenever $\F$ is a finite partition of $B$, there exist $F\in\F$ and $x_1,\dots,x_n\in B\subset S$ such that $$\sum_{t=1}^{n}x_{t}\in FS(\{x_t\}_{t=1}^{n})\subset F\subset B\subset S,$$ which contradicts with the hypothesis.
\end{proof}

In the analogous way, we have
\begin{lem}\label{IP_<omega*}
	Let $d\in\N$ and  $S=\{s_i:s_i<s_{i+1}\}$ be an infinite sequence of $\N$. If $\sum_{i=1}^{n}s_{k_i}\notin S$ for any $1\leq k_1<\dots<k_n$, $n\in\N$. Then $\N\backslash S$ is an IP$_{<\omega}$*-set.
\end{lem}

\begin{thm}\label{IP_n*}
	Let $d\in\N$ and $p_1,\dots,p_d$ be integral polynomials with $p_i(0)=0$ and $\max\{\deg p_{i}\}\geq2$. Then there exists an IP$_n$*-set (resp. IP$_{<\omega}$*-set) $A$ in $\N$ such that the set $$\{(m,n)\in\N^2:m+p_1(n),\dots,m+p_d(n)\in A\}$$ is not an IP$_n$*-set (resp. IP$_{<\omega}$*-set) in $\N^2$.
\end{thm}
\begin{proof}
	It is enough to prove the case of $d=1$ with $p(n)=\sum_{i=1}^{l}a_{i}n^i$, $l\geq 2, a_{l}>0$.	
	Take $\beta\in\N$ such that
	$p(\alpha),p'(\alpha)>0$,  and $p(2^{2^\alpha}+2)>p(2^{2^\alpha})+p(2^{2^1}+\dots+2^{2^{\alpha-1}})$ whenever $\alpha>\beta$.

	\medskip
	\noindent {\bf Claim.} $A:=\N\backslash S$ is an IP$_2$*-set (so an IP$_n$*-set, IP$_{>\omega}$*-set) where $$S=\{s_n:s_n<s_{n+1}\}:=\{p(n)+n:n=2^{2^{i_1}}+\dots+2^{2^{i_k}}, \beta<i_{1}<\dots<i_{k}, k\in\N\}.$$
	
	If the Claim holds, we have $$FS\{(m_i,n_i)_{i=1}^\infty\}\cap \{(m,n)\in\N^2:p(n)+m\in A\}=\varnothing$$ where $(m_i,n_i)=(2^{2^{i+\beta}},2^{2^{i+\beta}})$.
	Thus, 
	$\{(m,n)\in\N^2:p(n)+m\in A\}$ is not an IP*-set of $\N^2$, moreover, not an IP$_n$*-set (resp. IP$_{<\omega}$*-set).
	
	\medskip
	Now we are going to prove the Claim.
	Suppose $a=s_{n_t}<b=s_{n_r}$ where $$a=p(\sum_{i=1}^t 2^{2^{c_i}})+\sum_{i=1}^t 2^{2^{c_i}},b=p(\sum_{i=1}^r 2^{2^{d_i}})+\sum_{i=1}^r 2^{2^{d_i}} $$ for some $\beta<c_1<\dots<c_t ,\beta<d_1<\dots<d_r$. There are three cases: 

	\medskip
	\noindent {\bf Case 1:} $c_t=d_r$. In this case, we have  $$p(\sum_{i=\beta+1}^{c_t}2^{2^{i}})+\sum_{i=\beta+1}^{c_t}2^{2^{i}}<a+b<p(2^{2^{c_t+1}})+2^{2^{c_t+1}}.$$
	So there exists $n\in\N$ such that $$s_{n}=p(\sum_{i=\beta+1}^{c_t}2^{2^{i}})+\sum_{i=\beta+1}^{c_t}2^{2^{i}}, s_{n+1}=p(2^{2^{c_t+1}})+2^{2^{c_t+1}}.$$
	
	\medskip
	\noindent {\bf Case 2:}  $c_t<d_r$ with  $\{d_1,\dots,d_r\}\ne\{\beta+1,\beta+2,\dots,d_r\}$. In this case, it follows from the choice of $\beta$ we have $$s_{n_r}=b<a+b<p(2^{2^{d'}}+\sum_{i=1}^r 2^{2^{d_i}})+2^{2^{d'}}+\sum_{i=1}^r 2^{2^{d_i}}=s_{n_r+1}$$ where $d'=\min(\{1,2,\dots,d_r\}\backslash\{d_1,\dots,d_r\})$.
	
	\medskip
	\noindent {\bf Case 3:}  $c_t<d_r$ with $\{d_1,\dots,d_r\}=\{\beta+1,\beta+2,\dots,d_r\}$. In this case,  $$s_{n_r}=b<a+b<p(2^{2^{d_r+1}})+2^{2^{d_r+1}}=s_{n_r+1}.$$
	
	Among all cases above, $s_n<a+b<s_{n+1}$ for some $n\in\N$, which deduces that $a+b\notin S$. So $\N\backslash S$ is an IP$_2$*-set by Lemma \ref{IPn*}.
\end{proof}

\subsection{The case for $\Delta$* and central* sets
}\
\medskip
	
	In this subsection, we prove Theorem \ref{main} when ``large'' being ``$\Delta$*'' and  ``central*''.

	\begin{lem}\label{Delta**}
		If $S\subset\N$ (resp. $\Z$) doesn't contain any $\Delta$-set, then $\N\backslash S$ (resp. $\Z\backslash S$) is a $\Delta$*-set.
	\end{lem}
\begin{proof}
	If $\N\backslash S$ isn't a $\Delta$*-set, there exists a $\Delta$-set $B$ such that $B\cap (\N\backslash S)=\varnothing$. Then we have $B\subset S$, a contradiction.
\end{proof}

	\begin{lem}\label{Delta-con}
		Let $p(n)$ be an integral polynomial with $\deg p\geq 2$.
		Then $\{p(n):n\in\Z\}$ doesn't contain any $\Delta$-set.
	\end{lem}
\begin{proof}
	Assume that $F=\{f_i\}$ be an infinite set such that $\Delta(F)\subset \{p(n):n\in\N\}$.
	Then for $i<j$, we have $f_j-f_i=p(n_{i,j})$ for some $n_{i,j}\in\Z$. So $p(n_{i,k})=p(n_{i,j})+p(n_{k,j}),\forall 1\leq i<j<k$.
	
	Suppose $p(n)=\sum_{i=1}^{l}a_{i}n^i$ with $l\geq 2, a_l>0$. When $n$ big enough, 
	the equation $p(n_{1,k})-p(n_{2,k})=n_{1,2}$ has at most finitely many solutions. So $|\{p(n_{1,k}):k\in\N\}|<\infty$. Thus, we have
	$$F=\{f_k\}\subset \{f_1+p(n_{1,k}):k\in\N\}$$ is finite, this is a contradiction with the assumption.
\end{proof}



\begin{thm}\label{Delta*}
	Let $d\in\N$ and $p_1,\dots,p_d$ be integral polynomials with $p_i(0)=0$ and $\max(\deg p_{i})\geq2$. There exists a $\Delta$*-set $A$ in $\N$ such that the set $$\{(m,n)\in\N^2:m+p_1(n),\dots,m+p_d(n)\in A\}$$ is not a $\Delta$*-set in $\N^2$.
\end{thm}
\begin{proof}
It is enough to prove the case when $d=1$ with $\deg p\geq 2$.	
Let $\{(m_i,n_i)\}_{i=1}^{\infty}$ be any infinite subsequence of $\N^2$ with $m_i=n_i$ and $n_i<n_j$ whenever $i<j$. 
Denote $$
\begin{aligned}
	S:&=\{p(n)+m:(m,n)\in\Delta(\{(m_i,n_i)\}_{i=1}^{\infty})\}\\
	&=\{p(n_j-n_i)+n_j-n_i:i<j\in\N\}\subset \{p(n)+n:n\in\N\}.
\end{aligned}$$

It follows from Lemmas \ref{Delta**} and \ref{Delta-con} that the set $A:=\Z\backslash S\supset\Z\backslash\{p(n)+n:n\in\Z\}$ is a $\Delta$*-set. But meanwhile, we have
$$\{(m,n)\in\N^2:p(n)+m\in A\}=\{(m,n)\in\N^2:p(n)+m\notin S\}\subset \N^2\backslash\Delta\{(m_i,n_i)\}.$$
i.e.,
$$\Delta(\{(m_i,n_i)\})\cap \{(m,n)\in\N^2:p(n)+m\in A\}=\varnothing,$$
which means that
$\{(m,n)\in\N^2:p(n)+m\in A\}$ is not a $\Delta$*-set of $\N^2$.
\end{proof}
	
\bigskip

Next, we solve the case for central* sets. 

\begin{thm}\label{central*}
	Let $d\in\N$ and $p_1,\dots,p_d$ be integral polynomials with $p_i(0)=0$ and $\max\{\deg p_{i}\}\geq2$. Then there exists  a central* set $A$ in $\N$ such that the set $$\{(m,n)\in\N^2:m+p_1(n),\dots,m+p_d(n)\in A\}$$ is not central* in $\N^2$.
\end{thm}
\begin{proof}
Suppose that $d=1,p(n)=\sum_{i=1}^{l}a_{i}n^{i}$ with $a_l>0,l\geq 2$.
We construct a special thick set first. 
We take $s_0\in\N$ such that $p(s_0)>1$. By induction we obtain $s_n$ satisfying 
$$
\left\{  
\begin{aligned} 
	&2(p(s_n)+s_n+n)<p(s_n+1)+s_n,  \\  
	&p(s_n+n)+s_n+n<2(p(s_n)+s_n)<p(s_{n+1})+s_{n+1}, \\  
	&p(s_n)+s_n+n<p(s_{n-1})+s_{n-1}+p(s_n)+s_n<p(s_n+1)+s_n.
\end{aligned}  
\right.  
$$
Let $$S:=\bigcup_{n\geq 0}[s_n,s_n+n]\times[s_n,s_n+n].$$
Then $S$ is a thick subset in $\N^2$, therefore a central set.

Denote $D:=\{p(n)+m:(m,n)\in S\}$.
By the construction of $S$, we have $a+b\notin D$ whenever $a,b\in D$. 
Hence, $D$ contains no IP-set.
It follows from Lemma \ref{IP*} that the set $A:=\N\backslash D$ is an IP*-set, therefore a central* set. 

Now we consider set $\{(m,n)\in\N^2:p(n)+m\in A\}.$ We have
$$\{(m,n)\in\N^2:p(n)+m\in A\}=\{(m,n)\in\N^2:p(n)+m\notin D\}\subset\N^2\backslash S.$$
i.e.,
$$S\cap\{(m,n)\in\N^2:p(n)+m\in A\}=\varnothing.$$
which means that $\{(m,n)\in\N^2:p(n)+m\in A\}$ is not a central* subset in $\N^2$.
\end{proof}


\medskip



\begin{thebibliography}{SS}
	
	\bibitem{B} M. Beiglböck, Arithmetic progression in abundance by combinatorial tools, {\it Proc. Amer. Math. Soc.}, {\bf 137} (2009), no. 12, 3981–3983.
	
	\bibitem{BH90} V. Bergelson and N. Hindman, Nonmetrizable topological dynamics and Ramsey theory, {\it Trans. Amer. Math. Soc.} {\bf 320} (1990), 293320, 293-320.
	
	\bibitem{BH} V. Bergelson and N. Hindman,  Partition regular structures contained in large sets are abundant, {\it J. Combin. Theory Ser. A},  {\bf 93} (2001), no. 1, 18–36.
	
	
	\bibitem{F} H. Furstenberg, {\it Recurrence in ergodic theory and combinatorial number theory.} M. B. Porter Lectures. Princeton University Press, Princeton, N.J., 1981.
	
	\bibitem{FG}
	H. Furstenberg and E. Glasner, Subset dynamics and van der Waerden's theorem,
	{\it Contemp. Math.}, {\bf 215} (1998), 197-203. 
	
	
	
	\bibitem{GH} W. Gottschalk and G. Hedlund, {\it Topological dynamics, American Mathematical Society Colloquium Publications}, Vol. 36. American Mathematical Society, Providence, R. I., 1955.
	
	
	\bibitem{HSY} W. Huang, S. Shao and X. Ye, Topological dynamical systems induced by polynimials and conbinatorial consquences,  arXiv:2301.07873, 2023.
	
	
	\bibitem{L1} A. Leibman,
	Multiple recurrence theorem for nilpotent group actions,  {\it Geom. Funct. Anal.}, {\bf 4} (1994), no. 6, 648–659.
	
	\bibitem{V} J. de Vries, {\it Elements of Topological Dynamics}, Kluwer Academic Publishers, Dordrecht, 1993.


	\bibitem{wu} Q. Wu, Some dynamical properties related to polynomials, {\it Proc. Amer. Math. Soc.} {\bf 152} (2024), no. 4, 1633–1646.
	
	
\end{thebibliography}
\end{document}